\documentclass[12pt]{amsart}
\usepackage{amsthm}
\usepackage{amsmath}
\usepackage{amssymb}
\usepackage{mathrsfs}
\usepackage{enumerate}
\usepackage{float}
\usepackage{url, multicol}
\usepackage[all]{xy}
\usepackage{graphicx, pgf}
\usepackage{hyperref}
\allowdisplaybreaks

\newcommand{\Q}{{\mathbb Q}}

\newcommand{\Z}{{\mathbb Z}}

\newcommand{\f}{{\mathfrak f}}

\newcommand{\R}{{\mathbb R}}

\newcommand{\OO}{{\mathcal O}}

\newcommand{\q}{{\mathfrak q}}
\renewcommand{\a}{{\mathfrak a}}
\newcommand{\p}{{\mathfrak p}}
\newcommand{\<}{\langle}
\renewcommand{\>}{\rangle}

\newcommand{\m}{{\mathfrak m}}

\DeclareMathOperator{\ord}{ord}

\newcommand{\Gal}{\mathrm{Gal}}

\renewcommand{\c}{\mathfrak{C}}

\title[Continued Fractions and Iwasawa Invariants]{Using Continued Fractions to Compute Iwasawa Lambda Invariants of Imaginary Quadratic Number Fields}
\author{Jordan Schettler}
\address{Department of Mathematics\\
South Hall, Room 6607\\
University of California\\
Santa Barbara, CA 93106-3080\\
\href{mailto:jcs@math.ucsb.edu}{jcs@math.ucsb.edu}}

\begin{document}
 \setcounter{tocdepth}{1}
\newtheorem{thm}{Theorem}
\newtheorem{conj}[thm]{Conjecture}
\newtheorem{prop}[thm]{Proposition}
\newtheorem{lemma}[thm]{Lemma}
\newtheorem{corollary}[thm]{Corollary}
\newtheorem*{fact}{Fact}
\theoremstyle{remark}
\newtheorem{rem}[thm]{Remark}
\theoremstyle{definition}
\newtheorem{defn}[thm]{Definition}
\newtheorem*{goal}{Goal}
\newtheorem{assume}{Assumption}
\renewcommand{\theassume}{\Alph{assume}}
\newtheorem{exam}[thm]{Example}
\numberwithin{equation}{section}

\begin{abstract}
Let $\ell>3$ be a prime such that $\ell \equiv 3 \pmod{4}$ and $\Q(\sqrt{\ell})$ has class number 1. Then Hirzebruch and Zagier noticed that the class number of $\Q(\sqrt{-\ell})$ can be expressed as $h(-\ell) = (1/3)(b_1+b_2 + \cdots +b_m) - m$ where the $b_i$ are partial quotients in the `minus' continued fraction expansion $\sqrt{\ell} = [[b_0; \overline{b_1, b_2, \ldots, b_m}]]$. 
For an odd prime $p \neq \ell$, we prove an analogous formula using these $b_i$ which computes the sum of Iwasawa lambda invariants $\lambda_p(-\ell)+\lambda_p(-4)$ of $\Q(\sqrt{-\ell})$ and $\Q(\sqrt{-1})$. In the case that $p$ is inert in $\Q(\sqrt{-\ell})$, the formula pleasantly simplifies under some additional technical assumptions.
\end{abstract}

\maketitle

\tableofcontents

\section{Notation and Assumptions}

Let $K$ be a real quadratic number field of discriminant $D$. Suppose
\begin{align}\label{factor}
D = D_1D_2
\end{align}
where $D_1$, $D_2$ are the discriminants of quadratic number fields $K_1$, $K_2$, respectively. We will frequently make the following assumption.
\begin{assume}\label{ass}
Suppose the class number of $K$ is 1 and that $D$ is divisible by a prime congruent to $3$ modulo $4$.
\end{assume}
\begin{rem}\label{remark}
Make Assumption \ref{ass}. Then $K$ has no units of negative norm and the factorization in Eq. \ref{factor} is unique (up to ordering of factors) with $D_1$, $D_2$ negative by classical genus theory. Without loss of generality, $-D_1$ is a prime congruent to $3$ modulo $4$ and $-D_2$ is either $4$, $8$, or a prime congruent to $3$ modulo $4$.
\end{rem}
For $i=1,2$ let $h(D_i)$ denote the class number of $K_i$. For a prime $p$, let $\lambda_p(D_i)$ denote the Iwasawa lambda invariant of the cyclotomic $\Z_p$-extension of $K_i$.
\begin{goal}
Under Assumption \ref{ass}, we want a formula for the sum of lambda invariants $\lambda_p(D_1)+\lambda_p(D_2)$ which is analogous to Hirzebruch and Zagier's formula for the product of class numbers $h(D_1)h(D_2)$ given in terms of the partial quotients in the `minus' continued fraction expansion of $(\delta+\sqrt{D})/2$ where $\delta \in \{0,1\}$ with $D\equiv \delta \pmod{4}$.
\end{goal}
To accomplish this goal, we first recall some computations of special values of partial zeta functions obtained by Kronecker limit formulas at $s=1$ or by the methods of Takuro Shintani at $s=0$. Then we relate these to special values of $L$-functions which can be alternatively given in terms of the arithmetic invariants $h(D_i)$ and $\lambda_p(D_i)$.
\section{Special Values of Partial Zeta Functions}
Suppose $\m = \m_0\m_{\infty}$ is a modulus of $K$ where we view $\m_0$ as an ideal in the ring of integers $\OO_K$ and we will always assume $\m_{\infty}$ is the product of both real places of $K$.
We denote the narrow ray class group associated to $\m$ by $C_\m$ as in \cite{milneCFT}. 
Consider the partial zeta function $\zeta(s,\c)$ associated to some $\c\in C_\m$, i.e., the meromorphic continuation of the sum
\begin{align*}
\mathop{\sum_{\a \in \c}}_{\a \subseteq \OO_K} \frac{1}{N(\a)^s}
\end{align*}
where $N(\a)$ denotes the absolute norm of $\a$. We have a Laurent expansion
\begin{align*}
\zeta(s,\c) = \frac{\kappa}{s-1} + \varrho(\c) + \varrho_1(\c)(s-1) + \varrho_2(\c)(s-1)^2+\cdots
\end{align*}
where $\kappa$ is a constant which depends on $\m$ but not on $\c$. Computations of $\varrho(\c)$ are called `Kronecker limit formulas' for real quadratic number fields because Leopold Kronecker first computed this quantity in the context of an imaginary quadratic number field. If $\chi$ is a nontrivial character on $C_\m$, the $L$-function
\begin{align*}
L(s, \chi) = \sum_{\c \in C_\m}\chi(\c)\zeta(s, \c)
\end{align*}
has special values at $s=1$, $0$ given by
\begin{align*}
&L(1, \chi) = \sum_{\c \in C_\m}\chi(\c)\varrho(\c), &L(0, \chi) =  \sum_{\c \in C_\m}\chi(\c)\zeta(0, \c).
\end{align*}
We will state some results which express $\varrho(\c)$, $\zeta(0,\c)$ in terms of continued fractions, and, in order to do so, we need a couple of lemmas which we do not prove here. See \cite{Zagi} and \cite{Zagi2}.
\begin{lemma}\label{contfrac}
Suppose $\alpha \in \R \backslash \Q$. There is a unique `minus' continued fraction expansion\footnote{This `minus' expansion is related to the usual `plus' continued fraction expansion $\alpha = [a_0; a_1, a_2, a_3, \ldots]$ where $a_0 \in \Z$ and $a_1, a_2, \ldots \in \Z_{>0}$: as sequences
\begin{align*} (b_0, b_1, b_2, \ldots) = (a_0+1, \underbrace{2, 2, \ldots, 2}_{a_1-1}, a_2+2, \underbrace{2, 2, \ldots, 2}_{a_3-1}, a_4+2, \ldots).  \end{align*}}
\begin{align*}\alpha = [[b_0; b_1, b_2, b_3, \ldots]] \mathrel{\mathop:}=  \lim_{k\rightarrow \infty} b_0-\cfrac{1}{b_1-\cfrac{1}{b_2-\cdots-\cfrac{1}{b_k}}} \end{align*}
where $b_0 \in \Z$ and $b_1, b_2, \ldots \in \Z_{>1}$. Moreover, this expansion is eventually periodic if and only if $\alpha$ is algebraic of degree $2$.
In particular, for $\delta\in \{0,1\}$ with $D\equiv \delta \mod{4}$,
\begin{align*}
\frac{\delta + \sqrt{D}}{2} = [[b_0; \overline{b_1, b_2, \ldots, b_m}]]
\end{align*}
where the bar signifies the repeating block of minimal length $m$; moreover, $b_m =  2b_0 - \delta$ and we have a palindrome
\begin{align*}
(b_1, b_2, \ldots, b_{m-1}) &= (b_{m - 1}, b_{m - 2}, \ldots, b_1).
\end{align*}
\end{lemma}
\begin{lemma}\label{omega}
Let $\c\in C_\m$. For each integral ideal $\mathfrak{a} \in \c$, there are totally positive $z, \varpi \in K$ such that
\begin{align*}
(z)\mathfrak{a}^{-1}\m_0 = \Z + \varpi \Z
\end{align*}
where $\varpi > 1 > \varpi' > 0$ with $\varpi' = $ Galois conjugate of $\varpi$. This condition on $\varpi$ ensures that its `minus' continued fraction is purely periodic of some minimal period $m$:
\begin{align*}
\varpi = [[\overline{b_0; b_1,b_2, \ldots, b_{m-1}}]].
\end{align*}
Moreover, the sequence $(b_0, b_1, \ldots, b_{m-1})$ is determined up to cyclic permutation by $\c$.
\end{lemma}
\subsection{Meyer's Theorem}
Curt Meyer studied the case $\m_0 = \OO_K$ in \cite{Meye}. He expressed $\varrho(\c)$ as an integral involving a logarithm of the Dedekind $\eta$-function using methods earlier applied by Erich Hecke to wide ideal classes. The advantage of considering narrow ideal classes is that if there are no units of negative norm, every wide ideal class is the disjoint union of two narrow ideal classes $\c \coprod \c^{\ast}$, and the transformation properties of the Dedekind-$\eta$ function can then be used to explicitly evaluate the difference $\varrho(\c)-\varrho(\c^{\ast})$ as $\pi^2/\sqrt{D}$ multiplied times an expression involving a Dedekind sum. Friedrich Hirzebruch and Don Zagier noticed that this expression could be written as a sum of partial quotients in a certain `minus' continued fraction.
\begin{defn}\label{thetadef}
We denote by $1+\m_0$ the set of $\alpha\in K^{\times}$ such that $\ord_\p(\alpha - 1) \geq \ord_\p(\m_0)$ for all prime ideals $\p$ of $\OO_K$ dividing $\m_0$. Define $\Theta = [(\theta)]\in C_{\m}$ where $\theta$ is any positive element of $1 +\m_0$ whose Galois conjugate $\theta'$ is negative.
For each $\c\in C_\m$, we take $\c^{\ast} = \c\Theta$.
\end{defn}
\begin{thm}[Meyer]\label{Meyer}
Suppose $\m_0 = \OO_K$ and that $K$ has no units of negative norm. In the notation of Lemma \ref{omega}, we have
\begin{align*}
\varrho(\c) - \varrho(\c^{\ast}) = \frac{\pi^2}{6\sqrt{D}}\sum_{k=1}^m (b_k-3).
\end{align*}
\end{thm}
\subsection{Yamamoto's Theorem}
Meyer's theorem is sufficient to derive the known formula for class numbers, and there are generalizations which compute $\varrho(\c)$, $\varrho(\c^{\ast})$ for an arbitrary $\m_0$. For example, Shuji Yamamoto proved such a Kronecker limit formula for narrow ray classes in \cite{Yama}; he further computed $\zeta(0,\c)$, $\zeta(0,\c^{\ast})$ using the methods of Shintani in \cite{Shin}. We will find it more convenient to use these computations at $s=0$ because they are rational numbers and have considerably simpler descriptions in the general case.
\begin{defn}\label{c_k}
Let $\c\in C_\m$. Choose $\a$, $z$, $\varpi=[[\overline{b_0; b_1, \ldots, b_{m-1}}]]$ as in Lemma \ref{omega}. There are unique rational numbers $c_{-2}, c_{-1} \in [0,1)$ such that
\begin{align}\label{initial}
c_{-2}-c_{-1}\varpi-z \in \Z + \varpi\Z,
\end{align}
so if we extend $b_k = b_{k+m}$ by periodicity, we may recursively define $c_k \in [0,1)$ for all integers $k \geq 0$ by
\begin{align}\label{recursive}
c_k = \left\{b_kc_{k-1} - c_{k-2}\right\}
\end{align}
where $\left\{x\right\} = x - \lfloor x\rfloor$ denotes the fractional part of a real number $x$.
\end{defn}
\begin{defn}\label{unit}
Let $U_+$ denote the totally positive units in $\OO_K$ with generator $\varepsilon >1$. Take $\varepsilon_\m$ to be the unique generator of $U_+\cap(1+\m_0)$ which is greater than $1$, so $\varepsilon_\m = \varepsilon^r$ for some nonnegative integer $r$.
\end{defn}
\begin{thm}[Yamamoto]\label{Yamamoto}
Let $\c \in C_\m$. Then
\begin{align}\label{value}
\zeta(0,\c)= -\zeta(0,\c^{\ast}) = \sum_{k=1}^{rm} \left(\frac{b_k}{2}B_2(c_{k-1}) - B_1(c_{k-1})B_1(c_{k-2})  \right)
\end{align}
where $B_1(x) = x-\frac{1}{2}$, $B_2(x) = x^2-x+\frac{1}{6}$ are Bernoulli polynomials and $b_k$, $m$, $c_k$, $r$ are as in Definitions \ref{c_k}, \ref{unit}.
\end{thm}
\section{The Formula for Class Numbers}
In this section, we use the results stated above to outline a proof of the formula for class numbers due to Hirzebruch and Zagier. We do this in order to motivate the formula for Iwasawa lambda invariants.

We say $\chi$ is a genus character when $\chi$ is a real valued character on $C_\m$ with $\m_0=\OO_K$. In this case, we can use either Meyer's Theorem \ref{Meyer} or Yamamoto's Theorem \ref{Yamamoto} to compute class numbers by factoring $L(s,\chi)$ into the product of Dirichlet $L$-functions.
\begin{thm}[Kronecker]\label{Kronecker}
Let $t$ denote the number of distinct prime factors of $D$. Then there are exactly $2^{t-1}-1$ ways to factor $D=D_1D_2$ up to order as in Eq. \ref{factor}. Such factorizations are in bijection with the set of nontrivial genus characters $\chi$. Under this correspondence,
\begin{align}\label{factorization}
L(s, \chi) = L(s, \epsilon_1)L(s, \epsilon_2)
\end{align}
where each $\epsilon_i$ is the quadratic character of $K_i/\Q = \Q(\sqrt{D_i})/\Q$.
\end{thm}
\begin{thm}[Hirzebruch] Make Assumption \ref{ass}. Then
\begin{align}\label{classnumber}
h(D_1)h(D_2) =  \frac{w_1w_2}{24} \sum_{k=1}^m(b_k-3) 
\end{align}
where $(\delta+\sqrt{D})/2 = [[b_0; \overline{b_1, b_2, \ldots, b_m}]]$ as in Lemma \ref{contfrac} and each $w_i$ is the number of roots of unity in $K_i$.
\end{thm}
\begin{proof}[Proof]
Take $\m_0 = \OO_K$. By Remark \ref{remark} and Theorem \ref{Kronecker}, there is a unique nontrivial genus character $\chi$ on $C_\m$, so Eq. \ref{factorization} and the analytic class number formula imply
\begin{align}\label{intermediate}
L(1, \chi) = \frac{2\pi h(D_1)}{w_1\sqrt{-D_1}} \cdot \frac{2\pi h(D_2)}{w_2\sqrt{-D_2}}= \frac{4\pi^2}{w_1w_2\sqrt{D}}h(D_1)h(D_2)
\end{align}
or via functional equations
\begin{align}\label{intermediate2}
L(0, \chi) = \frac{2 h(D_1)}{w_1} \cdot \frac{2 h(D_2)}{w_2}= \frac{4}{w_1w_2}h(D_1)h(D_2).
\end{align}
Here we have simply $C_\m = \{\Theta, \Theta^{\ast}\}$. Consider the trivial class $\c = \Theta^{\ast}$ in the context of Lemma \ref{omega}: we can choose $\a = \OO_K$, $z=1$, and
\begin{align*}
\varpi = \frac{2b_0 -\delta + \sqrt{D}}{2} = [[\overline{2b_0 - \delta, b_1, b_2, \ldots, b_{m-1} }]]
\end{align*}
via Lemma \ref{contfrac}. It is also clear that $c_{-2} = c_{-1} = 0$ in Lemma \ref{c_k}, so $c_k=0$ for all $k$. Thus Meyer's Theorem \ref{Meyer} implies
\begin{align}\label{one}
L(1,\chi) = \varrho(\Theta^{\ast}) - \varrho((\Theta^{\ast})^{\ast}) = \frac{\pi^2}{6\sqrt{D}}\sum_{k=1}^m(b_k-3)
\end{align}
and Yamamoto's Theorem \ref{Yamamoto} implies
\begin{align}\label{two}
L(0,\chi) = \zeta(0,\Theta^{\ast}) - \zeta(0,(\Theta^{\ast})^{\ast})= \frac{1}{6}\sum_{k=1}^m(b_k-3).
\end{align}
Combining either Eq. \ref{one} with Eq. \ref{intermediate} or Eq. \ref{two} with Eq. \ref{intermediate2} will both yield the desired result.
\end{proof}

\section{The Formula for Iwasawa Lambda Invariants}\label{abovesection}

Fix a prime $p$ and number field $F$. Let $F_{\infty}$ denote the cyclotomic $\Z_p$-extension\footnote{We will not consider any non-cyclotomic $\Z_p$-extensions in this paper.} of $F$, i.e., $F_{\infty}$ is the unique subfield of
\begin{align*}\bigcup_{n\geq 1}F(\zeta_{p^n}) \subseteq\overline{\Q}\end{align*}
such that $\Gal(F_{\infty}/F)$ is isomorphic to the group $\Z_p$ of $p$-adic integers where $\overline{\Q}$ is some fixed algebraic closure and each $\zeta_{p^n}$ a primitive $p^n$th root of unity. The subfields of $F_{\infty}$ which contain $F$ all lie in a tower
\begin{align*}
F \subset F_1 \subset F_2 \subset \ldots \subset F_{\infty}
\end{align*}
where
\begin{align*}
\Gal(F_n/F) \cong \Z/(p^n) \mbox{ for all } n \geq 1.
\end{align*}
The $p$-parts of the class numbers of these intermediate fields become regularly behaved.
\begin{thm}[Iwasawa's Growth Formula]
There are integers $\lambda_p(F)$, $\mu_p(F)$, $\nu_p(F)$ such that class numbers $h_n$ of $F_n$ satisfy
\begin{align}\label{growth}
\ord_p(h_n) = \lambda_p(F)n + \mu_p(F)p^n+\nu_p(F)
\end{align}
for all sufficiently large $n$ where $\ord_p$ denotes the $p$-adic order.
\end{thm}
Here is a short list of what is known and conjectured about the Iwasawa invariants $\lambda$, $\mu$, $\nu$ which appear in the growth formula:
\begin{itemize}
\item Iwasawa conjectured that $\mu_p(F)=0$ for all $p$ and $F$. 
\item Lawrence Washington and Bruce Ferrero proved that $\mu_p(F)=0$ for all $p$ when $F/\Q$ is abelian (see \cite{Ferr2}).
\item If $F$ has only one prime lying over $p$ and $p$ does not divide the class number of $F$, then $\lambda_p(F)=\mu_p(F)=\nu_p(F)=0$ (see \cite{Iwas5}).
\item If $p$ splits completely in $F$, then $\lambda_p(F)\geq r_2$ where $r_2$ is the number of complex places of $F$ (see, e.g., \cite{Gree2}).
\item Greenberg conjectured that $\lambda_p(F)=0$ for all primes $p$ when $F$ is a totally real number field (see \cite{Gree3}).
\end{itemize}
Suppose now that $F$ is a quadratic number field of discriminant $\Delta$, and write
\begin{align*}
\lambda_p(\Delta) \mathrel{\mathop :}= \lambda_p(F),\hspace{0.2 in} \mu_p(\Delta) \mathrel{\mathop :}= \mu_p(F).
\end{align*}
Thus we always have $\mu_p(\Delta) = 0$, and conjecturally $\lambda_p(\Delta) = 0$ when $\Delta>0$. Assume now that $\Delta<0$. Then $\lambda_p(\Delta) \geq 1$ for infinitely many primes $p$. In fact, it is conjectured that $\lambda_p(\Delta)$ is bounded for fixed $\Delta$ and unbounded for fixed $p$. Bruce Ferrero (see \cite{Ferr}) and Y\^{u}ji Kida (see \cite{Kida2}) proved that for $-4\neq \Delta \neq -8$ we have
\begin{align}\label{above}
\lambda_2(\Delta) = -1 + \mathop{\sum_{\ell | \Delta}}_{\ell\neq 2} 2^{\ord_2(\ell^2-1) - 3}
\end{align}
where the sum ranges over all odd primes $\ell$ dividing $\Delta$. In particular, this shows that $\lambda_2(\Delta)$ is unbounded. For odd $p$, there seems to be no simple formula like \ref{above} to compute $\lambda_p(\Delta)$. We will derive a formula for $\lambda_p(D_1)+\lambda_p(D_2)$ under Assumption \ref{ass} which is analogous to the formula \ref{classnumber} for class numbers. We first need to recall how the lambda invariant in the growth formula \ref{growth} is related to special values of $L$-functions. We assume here that $p\nmid \Delta$ and $p$ is odd for simplicity. Let $\epsilon$ denote the quadratic character for $F/\Q = \Q(\sqrt{\Delta})/\Q$. For each integer $n\geq 1$, choose a Dirichlet character $\psi_n$ which generates the $n$th level $\Q_n \subseteq \Q(\zeta_{p^{n+1}})$ in the cyclotomic $\Z_p$-extension of $\Q$. In particular, $\psi_n$ has conductor $p^{n+1}$ and order $p^n$. By a theorem of Kubota and Leopoldt, there is a $p$-adic analytic function $L_p(s, \epsilon\psi_n\omega)$ on the disk $|s| < p^{(p-2)/(p-1)}$ in $\mathbb{C}_p$ such that
\begin{align*}
L_p(1-m, \epsilon\psi_n\omega) = (1-\epsilon\psi_n\omega^{1-m}(p)p^{m-1})L(1-m, \epsilon\psi_n \omega^{1-m})
\end{align*}
for all integers $m \geq 1$ where $\omega$ is the Teichm\"{u}ller character. In fact, there is an interpolating power series $f(T, \epsilon\omega) \in \Z_p[[T]]$ such that
\begin{align*}
L_p(s, \epsilon\psi_n\omega) = f(\zeta_{p^n}(1-p\Delta)^s-1, \epsilon\omega)
\end{align*}
for all $s\in \Z_p$ where $\zeta_{p^n} = \psi_n(1-p\Delta)^{-1}$ is a primitive $p^n$th root of unity. Setting $0=1-m=s$ in the previous two equations gives
\begin{align}\label{f}
f(\zeta_{p^n}-1, \epsilon\omega) = L_p(0,  \epsilon\psi_n\omega) = L(0, \epsilon\psi_n).
\end{align}
We define lambda and mu invariants of the power series
\begin{align*}
f(T, \epsilon\omega)= a_0+a_1T+a_2T^2+a_3T^3+\cdots
\end{align*}
as follows
\begin{align*}
\mu(f) &\mathrel{\mathop:}= \min\{\ord_p(a_i): i\geq 0\}  \\
\lambda(f) &\mathrel{\mathop:}= \min\{i\geq 0 : \ord_p(a_i) = \mu(f) \}.
\end{align*}
On can use Eq. \ref{f} to prove the growth formula \ref{growth} for $F = \Q(\sqrt{\Delta})$ (see \cite{Sinn}), and, in fact,
\begin{align*}
\mu(f) &= 0 \\
\lambda(f) &= \lambda_p(\Delta).
\end{align*}
Here we are using the assumption that $p$ is odd; we get a different computation for $\mu(f)$ when $p=2$.
We compute
\begin{align*}
\ord_p L(0, \epsilon\psi_n) &= \ord_p (a_0+a_1(\zeta_{p^n}-1)+a_2(\zeta_{p^n}-1)^2 + \cdots ) \\
&\geq \min\{\ord_p(a_i(1-\zeta_{p^n})^i) :i\geq 0 \} \\
&= \min\left\{\ord_p(a_i) +\frac{i}{\varphi(p^n)} : i\geq 0 \right\}
\end{align*}
where $\varphi$ is the Euler totient function. The inequality is an equality if the minimum is assumed by exactly one member of the set. In particular, 
\begin{align*}
\ord_p L(0, \epsilon\psi_n) = \frac{\lambda_p(\Delta)}{\varphi(p^n)}
\end{align*}
whenever $\varphi(p^n)>\lambda_p(\Delta)$. Note that we always have
\begin{align*}
\ord_p (a_0+a_1(\zeta_{p^n}-1)+\cdots + a_{\lambda_p(\Delta)-1}(\zeta_{p^n}-1)^{\lambda_p(\Delta)-1}) \geq 1
\end{align*}
and
\begin{align*}
\ord_p (a_{\lambda_p(\Delta)}(\zeta_{p^n}-1)^{\lambda_p(\Delta)} +  a_{\lambda_p(\Delta)+1}(\zeta_{p^n}-1)^{\lambda_p(\Delta)+1} + \cdots  ) = \frac{\lambda_p(\Delta)}{\varphi(p^n)}.
\end{align*}
Thus letting $\p_n=(1-\zeta_{p^n})$ denote the unique prime ideal lying above $p$ in $\Z[\zeta_{p^n}]$, we get
\begin{align}\label{ideal}
\notag \lambda_p(\Delta) &= \varphi(p^n)\ord_{p} L(0, \epsilon\psi_n) \\
&= \ord_{\p_n} L(0, \epsilon\psi_n)
\end{align}
whenever $\ord_{\p_n} L(0, \epsilon\psi_n)< \varphi(p^n)$. There is a partial converse to this statement which follows from the same observations; namely, if $\ord_{\p_n} L(0, \epsilon\psi_n) \geq \varphi(p^n)$, then we must also have $\lambda_p(\Delta) \geq \varphi(p^n)$.

At this point, we should remark that the special values $L(0, \epsilon\psi_n)$ can be computed with generalized Bernoulli numbers via the formula
\begin{align*}
L(1-m, \epsilon\psi_n) = -\frac{B_{m, \epsilon\psi_n}}{m} \mbox{ for all integers }m \geq 1.
\end{align*}
In particular,
\begin{align*}
L(0, \epsilon\psi_n) &= -B_{1,\epsilon\psi_n} = -\frac{1}{\f}\sum_{a=1}^{\f}\epsilon(a)\psi_n(a)a.
\end{align*}
where $\f = -\Delta p^{n+1}$ is the conductor of $\epsilon\psi_n$. This shows that $L(0, \epsilon\psi_n)$ is an algebraic integer by the work of Carlitz in \cite{Carl}. However, we will compute this special value in a different way by using Yamamoto's Theorem \ref{Yamamoto}.

~

Factor $D = D_1D_2$ as in Eq. \ref{factor}, and suppose each $D_i<0$. Then Eq. \ref{ideal} implies that for sufficiently large $n$ (which we fix for the following discussion) we have
\begin{align}\label{firstpart}
\lambda_p(D_1)+\lambda_p(D_2) &= \ord_{\p_n}L(0, \epsilon_1\psi_n) + \ord_{\p_n} L(0, \epsilon_2\psi_n)\\
&=\ord_{\p_n} L(0, \chi_n) \notag \\
&= \ord_{\p_n} \sum_{\c \in C_\m} \chi_n(\c) \zeta(0,\c) \notag
\end{align}
where
\begin{align*}
L(s, \chi_n) = L(s, \epsilon_1\psi_n)L(s, \epsilon_2\psi_n)
\end{align*}
is the $L$-function for a character $\chi_n$ on the narrow ray class group $C_\m$ of the real quadratic number field $K = \Q(\sqrt{D})$ with modulus $\m = (p^{n+1})\m_{\infty}$. For a prime ideal $\q$ of $\OO_K$ with $\q \cap \Z = (q)$ and $q \neq p$ we have
\begin{align}\label{defn}
\chi_n(\q) = \chi(\q)\psi_n(q^f)
\end{align}
where $\chi$ is the nontrivial genus character associated to the factorization $D = D_1D_2$ and $f$ is the residue degree of $\q/q$. Thus for a nonzero ideal $I$ in $\OO_K$ we have $\chi_n(I) = \chi(I)\psi_n(N(I))$ where $N(I)$ is the absolute norm of $I$, so $\chi_n((a)) = \psi_n(a^2)$ for all $a\in\Z$. Suppose now that $D$ is divisible by a prime congruent to 3 modulo 4. Then the narrow ray class group $C_\m$ is an internal direct product $C_\m^+ \times \<\Theta\>$ where $\Theta$ is as in Definition \ref{thetadef} and $C_\m^+$ is the kernel of the natural homomorphism $C_\m \rightarrow C_{\m_{\infty}} \cong \Z/(2)$. We have  $\chi_n(\Theta) = -1$, so $\chi_n(\c)\zeta(0,\c) = \chi_n(\c^{\ast})\zeta(0,\c^{\ast})$ for all $\c\in C_\m$. Thus since $\ord_{\p_n}(2)=0$ we get
\begin{align}\label{C+}
\lambda_p(D_1)+\lambda_p(D_2) = \ord_{\p_n} \sum_{\c \in C_\m^+} \chi_n(\c) \zeta(0,\c).
\end{align}
If, additionally, the class number of $K$ is 1, we have an exact sequence
\begin{align}\label{exactseq}
U_+ = \<\varepsilon\>\rightarrow (\OO_K/(p^{n+1}))^{\times} \rightarrow C_\m^+ \rightarrow 0
\end{align}
where the first map sends the fundamental unit $\varepsilon$ to its congruence class $\overline{\varepsilon}$ modulo $p^{n+1}$ and the second map sends the congruence class $\overline{\alpha}$ modulo $p^{n+1}$ of a totally positive $\alpha \in \OO_K$ to the class $[(\alpha)]\in C_{\m}^+$ of the principal ideal $(\alpha) \subseteq \OO_K$. (Note that every congruence class modulo $p^{n+1}$ has a totally positive representative, and any two such representatives for the same congruence class will generate the same narrow ray class.) Consider such a class $[(\alpha)]$ in the context of Lemma \ref{omega} with $\m_0=(p^{n+1})$; we may take $\a = (\alpha)$ and $z = \alpha/p^{n+1}$ so that
\begin{align*}
(z)\a^{-1}\m_0 = \OO_K = \Z+\varpi\Z
\end{align*}
with
\begin{align*}
\varpi = \frac{2b_0 -\delta + \sqrt{D}}{2} = [[\overline{2b_0 - \delta, b_1, b_2, \ldots, b_{m-1} }]]
\end{align*}
where $(\delta+\sqrt{D})/2 = [[b_0, \overline{b_1, \ldots, b_m}]]$ as in Lemma \ref{contfrac}. Write
\begin{align*}
\alpha = x + y \frac{\delta+\sqrt{D}}{2}
\end{align*}
with $x,y\in\Z$. Define
\begin{align*}
c_{-2} = \left\{ \frac{x-(b_0-\delta)y}{p^{n+1}}\right\}  \hspace{0.1 in}\mbox{ and } \hspace{0.1 in} c_{-1} = \left\{\frac{-y}{p^{n+1}}\right\},
\end{align*}
so that condition \ref{initial} is satisfied. As per Eq. \ref{recursive}, we have
\begin{align*}c_0 = \{(2b_0-\delta)c_{-1}-c_{-2}\} =  \left\{\frac{-(x+b_0y)}{p^{n+1}}\right\},\end{align*}
and it follows that for all $k \geq -1$
\begin{align}\label{hterm}
c_{k} = \left\{\frac{-(xq_k+yp_k)}{p^{n+1}}\right\}
\end{align}
where $p_k$ and $q_k$ are the numerator and denominator, respectively, of the $k$th convergent $[[b_0; b_1, \ldots, b_k]]$ for $(\delta+\sqrt{D})/2$ with $p_{-1} = 1$, $q_{-1} = 0$ by convention. Then Yamamoto's Theorem \ref{Yamamoto} implies
\begin{align}\label{prep}
\zeta(0,[(\alpha)]) = \sum_{k=1}^{r_nm} \left(\frac{b_k}{2}B_1(c_{k-1})^2 - B_1(c_{k-1})B_1(c_{k-2}) \right)+ C
\end{align}
where $r_n = \log(\varepsilon_\m)/\log(\varepsilon)$ is the order of $\varepsilon$ modulo $p^{n+1}$ and
\begin{align*}
C =  -\frac{r_n}{24}\sum_{k=1}^{m} b_k
\end{align*}
does not depend on the class $[(\alpha)]$. 

Let $g\in \Z$ be a primitive root modulo all powers of $p$, so, in particular, $\overline{g}$ has order $p^n(p-1)$ in $(\OO_K/(p^{n+1}))^{\times}$. Let $r_0$ denote the order of $\varepsilon$ modulo $p$. Then $r_0|p\pm1$ where the sign is $+$ or $-$ when $p$ is inert or split, respectively, in $K$. We have an isomorphism of abelian groups
\begin{align*}
(\OO_K/(p^{n+1}))^{\times} \cong \left\{\begin{array}{ll} \Z/(p^n(p^2 - 1)) \oplus \Z/(p^n) & \mbox{$p$ inert in $K$} \\
\Z/(p^n(p - 1)) \oplus \Z/(p^n(p-1)) & \mbox{$p$ split in $K$.} 
  \end{array}\right.
\end{align*}
Note that if $\varepsilon^a \equiv b \pmod{p^{n+1}}$ for some integers $a,b$, then we get a congruence of norms $b^2 = N(b) \equiv N(\varepsilon)^a = 1 \pmod{p^{n+1}}$, so $b$ is either $1$ or $-1$ modulo $p^{n+1}$. Hence the subgroup $\<\overline{\varepsilon}\>\cap\< \overline{g}\>\subseteq (\OO_K/(p^{n+1}))^{\times}$ has order $2$ or $1$ depending on whether there does or does not, respectively, exist an integer $c$ such that $-1\equiv \varepsilon^c \pmod{p^{n+1}}$; the existence of such a $c$ is equivalent to the statement that $2|r_0$. We will often make the following simplifying assumption.
\begin{assume}\label{bass}
Suppose that $p^2\nmid \varepsilon^{r_0}-1$. \footnote{Of course, the statement in the assumption does not always hold; e.g., if $p=7$ and $D = 23\cdot4$, then $\varepsilon = 24 + 5\sqrt{23}$ has order $3$ modulo 7, and, in fact, $7^2|\varepsilon^3-1$.}
\end{assume}
Making Assumption \ref{bass} implies $r_n = p^nr_0$ is the order of $\varepsilon$ modulo $p^{n+1}$, and thus the quotient group $(\OO_K/(p^{n+1}))^{\times}/\<\overline{\varepsilon}, \overline{g}\>$ is cyclic of $p$-prime order $v = 2^u(p \pm 1)/r_0$ where $u=1$ if $2|r_0$ and $u=0$ otherwise. Choose a totally positive\footnote{In fact, we may choose \emph{any} $\eta$ whose congruence class generates this quotient since the quantities $h_{j,k}$ can be modified modulo $p^{n+1}$.} $\eta \in \OO_K$ whose congruence class $\overline{\eta}$ generates this quotient. Then we have a surjection
\begin{align*}
\{ (i,j) : 1\leq i\leq p^n(p-1), 0\leq j \leq v-1 \} \rightarrow C_\m^+
\end{align*}
given by $(i,j) \mapsto [(g^i\eta^j)]$ which is either one-to-one or two-to-one depending on whether $u=0$ or $u=1$, respectively. For each $j$ write
\begin{align*}
\eta^j = x_j+y_j \frac{\delta+\sqrt{D}}{2}
\end{align*}
where $x_j, y_j \in \Z$, and then define
\begin{align*}
h_{j,k} = -(x_jq_k+y_j p_k)
\end{align*}
for each $k$. Note that $\chi_n( [(g^i\eta^j)]) = \psi_n(g^{2i})$ since $\eta$ has prime-to-$p$ order modulo $p^{n+1}$, so Eq.s \ref{C+} and \ref{prep} imply that
\begin{align*}
\lambda_p(D_1)+\lambda_p(D_2) = &\ord_{\p_n} \sum_{i=1}^{p^n(p-1)}\sum_{j=0}^{v-1}\sum_{k=1}^{r_nm} \psi_n^2(g^{i}) \left(\frac{b_{k}}{2}B_1\left( \left\{\frac{h_{j,k-1} g^i}{p^{n+1}}\right\} \right)^2\right. \\
&\left.- B_1\left( \left\{\frac{h_{j,k-1}g^i}{p^{n+1}}\right\} \right)B_1\left( \left\{\frac{h_{j,k-2}g^i}{p^{n+1}}\right\} \right)^{\textcolor{white}{2}} \!\!\!\! \right).
\end{align*}
To ease notation we define a \emph{twisted}, homogeneous Dedekind sum for an arbitrary Dirichlet character $\psi$ of modulus $\f$:
\begin{align*}
D_{\psi}(a,b) = \sum_{t=1}^{\f} \psi(t)\left( \left\{\frac{at}{\f}\right\} - \frac{1}{2}\right)\left( \left\{\frac{bt}{\f}\right\}-\frac{1}{2}\right).
\end{align*}
Since the character $\psi_n^2$ also generates the $n$th level in the cyclotomic $\Z_p$-extension of $K$ and since $t = g^i$ runs through the units modulo $p^{n+1}$ as $i$ runs though $\{1, 2, \ldots, p^n(p-1)\}$, we have proved the following.
\begin{thm}\label{mainthm}
Make Assumption \ref{ass}. Suppose $p\nmid D$ is an odd prime satisfying Assumption \ref{bass}. Then for sufficiently large $n$
\begin{align}\label{mainformula}
\lambda_p(D_1)+\lambda_p(D_2) = \ord_{\p_n} 
\sum_{j,k}\left(\frac{b_k}{2}D_{\psi_n}(h_{j,k-1}, h_{j,k-1}) - D_{\psi_n}(h_{j,k-1},h_{j,k-2}) \right)
\end{align}
where $(\delta+\sqrt{D})/2 = [[b_0; \overline{b_1, b_2, \ldots, b_m}]]$ as in Lemma \ref{contfrac} and the rest of the notation is as above.
\end{thm}
\begin{rem}
We can compute the Dedekind sums $D_{\psi_n}(a,a)$ for any integer $a$ as follows. Write $a = p^ma'$ where $p\nmid a'\in \Z$ and $0\leq m \in\Z$. If $m\geq n+1$, then $D_{\psi_n}(a,a) = 0$ since $\{at/p^{n+1}\}=0$ for all $t\in\Z$. Thus we may assume $0\leq m\leq n$. Choose $b\in\Z$ with $a'b\equiv 1\pmod{p^{n+1-m}}$. Then
\begin{align*}
&D_{\psi_n}(a,a) = \sum_{t=1}^{p^{n+1}} \psi_n(bt)\left( \left\{\frac{a'bt}{p^{n+1-m}}\right\} - \frac{1}{2}\right)^2 \\
&= \psi_n(b) \mathop{\sum_{r=1}^{p^{n+1-m} -1}}_{p\nmid r}\left(\frac{r}{p^{n+1-m}}-\frac{1}{2}\right)^2 \, \sum_{q=0}^{p^{m}-1} \psi_n(qp^{n+1-m}+r).
\end{align*}
If $m=0$ (i.e., $p\nmid a$), then $a=a'$ and $ab\equiv 1 \pmod{p^{n+1}}$, so
\begin{align}\label{same}
D_{\psi_n}(a,a)=\overline{\psi}_n(a)\sum_{t=1}^{p^{n+1}} \psi_n(t) \frac{t^2}{p^{2n+2}}
\end{align}
since
\begin{align*}
&\sum_{t=1}^{p^{n+1}} \psi_n(t) \left\{\frac{t}{p^{n+1}}\right\} = \sum_{t=1}^{p^{n+1}-1} \psi_n(p^{n+1}-t)\left\{\frac{p^{n+1}-t}{p^{n+1}}\right\} \\
&= \sum_{t=1}^{p^{n+1}-1} \psi_n(t)\left(1-\left\{\frac{t}{p^{n+1}}\right\}\right) = -\sum_{t=1}^{p^{n+1}} \psi_n(t) \left\{\frac{t}{p^{n+1}}\right\} = 0.
\end{align*}
On the other hand, if $m>0$ (i.e., $p|a$), then for $g\in \Z$ a primitive root modulo all powers of $p$ as above, we have
\begin{align*}
\sum_{q=0}^{p^{m}-1} \psi_n(qp^{n+1-m}+r) = \sum_{j=0}^{p^m-1} \zeta_{p^n}^{i+jp^{n-m}(p-1)} = \zeta_{p^n}^i \sum_{j=0}^{p^m-1} \zeta_{p^m}^{j(p-1)} = 0
\end{align*}
where $g^i=r$ and $\psi_n(g) = \zeta_{p^n}$ is a primitive $p^n$th root of unity. Thus $D_{\psi_n}(a,a)=0$, so Equation \ref{same} holds in this case as well since $\psi_n(a)=0$ when $p|a$. We summarize the results of this remark in the following proposition.
\end{rem}
\begin{prop}
We have for all $a\in \Z$
\begin{align}
D_{\psi_n}(a,a) = \frac{\overline{\psi}_n(a)}{p^{2n+2}}\sum_{t=1}^{p^{n+1}-1}\psi_n(t) \cdot t^2.
\end{align}
\end{prop}
In light of the above, one might hope to also evaluate the sums
\begin{align*}
D_{\psi_n}(h_{j,0},h_{j,-1})+D_{\psi_n}(h_{j,1},h_{j,0}) + \cdots + D_{\psi_n}(h_{j,r_nm-1},h_{j,r_nm-2})
\end{align*}
using a reciprocity law for Dedekind sums with characters, but the author is presently unaware of how this can be done. Nonetheless, Theorem \ref{mainthm} provide us with a means of computing lambda invariants.
\begin{exam}
Take $p=3$ and let $\ell \equiv 11 \pmod{12}$ be a prime such that the number field $K=\Q(\sqrt{\ell})$ of discriminant $D = 4\ell$ has class number one.
Then the totally positive fundamental unit $\varepsilon > 1$ in $K$ has order dividing $4$ in $(\OO_K/(3))^{\times}$.
Suppose this order is exactly $4$ and that $3^2 \nmid \varepsilon^4 - 1$.
Then for any positive integer $n$, we have that $2$ is a primitive root modulo $3^{n+1}$ and
that $(\OO_K/(3^{n+1}))^{\times}/\<\overline{\varepsilon}, \overline{2}\>$ is cyclic of order $2$ since the order $r_n$ of
$\varepsilon$ modulo $3^{n+1}$ will be $r_n = 4\cdot 3^n$.
We want a generator $\eta \in \OO_K$ of this quotient, and it clearly suffices to choose $\overline{\eta}$ to be an element of order $8$ in
\[(\OO_K/(3^{n+1}))^{\times} \cong \frac{\Z}{(3^n)} \oplus \frac{\Z}{(3^n)} \oplus \frac{\Z}{(8)}.\]
Alternatively, we may regard $\eta$ as an eighth root of unity in $\Q_3(\sqrt{\ell})$, and in that case, a fixed choice of $\eta$ will suffice for all $n$. We can construct such an $\eta$ by using Hensel lifting on $1+\sqrt{\ell}$ since
\begin{align*}
(1+\sqrt{\ell})^4 = (1+\ell + 2\sqrt{\ell})^2 \equiv (-\sqrt{\ell})^2 = \ell \equiv -1 \pmod{3}.
\end{align*}
Let us consider now a concrete case. Take $\ell=239 \equiv 11 \pmod{12}$. Then
\[\sqrt{\ell} = \sqrt{239} = [[16, \overline{2, 7, 4, 2, 2, 2, 17, 2, 2, 2, 4, 7, 2, 32}]],\]
so
\begin{align*}
\varepsilon = p_{14-1}+q_{14-1}\sqrt{239} = 6195120 + 400729\sqrt{239}.
\end{align*}
It is easy to check that $\varepsilon$ has order $4$ modulo $3$ since
\begin{align*}
\varepsilon \equiv 0+1\sqrt{239} \pmod{3}
\end{align*}
and $239 \equiv -1 \pmod{3}$. We also easily verify that $9\nmid \varepsilon^4-1$ since
\begin{align*}
\varepsilon^4 \equiv (6 + 4\sqrt{239})^4 \equiv (8+3\sqrt{239})^2 \equiv 1+3\sqrt{239} \pmod{9}.
\end{align*}
For $n=1$, we compute the right hand side of Eq. \ref{mainformula} and get
\begin{align*}
\ord_{\p_1} (-12\zeta_3 - 24) = \ord_{\p_1} (3 (\zeta_3 -1 + 3)) = 2+1 = 3>2=\varphi(3)
\end{align*}
where $\p_1 = (\zeta_3-1)$. Of course, $\lambda_3(-4)=0$ since 3 remains inert in $\Q(i)$, so we must have $\lambda_3(-\ell)\geq 2$ by the comments following Eq. \ref{ideal}. Likewise, for $n=2$ we get
\begin{align*}
\ord_{\p_2}(72\zeta_{3^2}^5 + 12\zeta_{3^2}^4 + 72\zeta_{3^2}^3 + 84\zeta_{3^2}^2 - 12\zeta_{3^2} + 12) = 6 = \varphi(3^2)
\end{align*}
where $\p_2 = (\zeta_{3^2}-1)$, so $\lambda_3(-\ell)\geq 6$. For $n=3$, we find
\begin{align*}
 \ord_{\p_3}(&-60\zeta_{3^3}^{17} + 140\zeta_{3^3}^{16} + 212\zeta_{3^3}^{15} + 112\zeta_{3^3}^{14} - 40\zeta_{3^3}^{13} + 8\zeta_{3^3}^{11} \\
&- 36\zeta_{3^3}^{10} + 40\zeta_{3^3}^9 + 184\zeta_{3^3}^7 + 68\zeta_{3^3}^6 - 16\zeta_{3^3}^5 - 128\zeta_{3^3}^4 \\
 &- 92\zeta_{3^3}^3 + 136\zeta_{3^3}^2 + 96\zeta_{3^3} + 36) \\
&= 6 < 18 =\varphi(3^3)
\end{align*}
where $\p_3 = (\zeta_{3^3}-1)$. Thus $\lambda_3(-239) = 6$. This and other similar computations (with the help of gp/pari) agree with known results as found in \cite{Dumm} for example.
\end{exam}
Under additional assumptions, we can compute lambda invariants without having to compute an $\eta \in \OO_K$ as above. In particular, the invariants can be computed using only a choice of primitive root $g$ and the mod $p^{n+1}$ data from the continued fraction expansion of $\sqrt{\ell}$.
\begin{corollary}
Suppose $\ell \equiv 3 \pmod{4}$ is a prime such that $\Q(\sqrt{\ell})$ has class number $1$. For each $k$ let $p_k$ and $q_k$ denote the numerator and denominator, respectively, of the $k$th convergent in the `minus' continued fraction expansion $\sqrt{\ell} =[[b_0; \overline{b_1, \ldots, b_m}]]$ where $m=$ minimal period and $p_{-1} = 1$, $q_{-1} = 0$ by convention. Let $p\neq \ell$ be an odd prime such that $p$ is inert in $\Q(\sqrt{-\ell})$ and that the fundamental unit $\varepsilon = p_{m-1}+q_{m-1}\sqrt{\ell}$ has order $r$ modulo $p$ satisfying the following technical assumptions:
\begin{enumerate}
\item $p^2 \nmid \varepsilon^{r}-1$
\item $r = p+1$ if $p \equiv 1\pmod{4}$
\item $r = (p-1)/2$ if $p \equiv 3\pmod{4}$.
\end{enumerate}
Choose $g\in \Z$ to be a primitive root modulo all powers of $p$, so there are integers $e_1$, $e_2$ with $\ell \equiv g^{e_1} \pmod{p^{n+1}}$ and $2e_2\equiv e_1 \pmod{p^n}$. For all $i,k$ take $p_k^{(i)}$, $q_k^{(i)}$ to be the least nonnegative residues of $g^i p_k$, $g^i q_k$ modulo $p^{n+1}$, and let $s_k^{(i)}$, $t_k^{(i)}$ denote the unique integers such that $b_kp_{k-1}^{(i)}-p_{k-2}^{(i)} = p_k^{(i)}+s_k^{(i)}p^{n+1}$ and $b_kq_{k-1}^{(i)}-q_{k-2}^{(i)} = q_k^{(i)}+t_k^{(i)}p^{n+1}$. Then for sufficiently large $n$,
\begin{align*} \lambda_p(-\ell)+\lambda_p(-4) = \ord_{\p_n}\sum_{i=1}^{p^n(p-1)}  \zeta_{p^n}^i \sum_{k=1}^{p^nrm} \frac{ t_{k}^{(i)}q_{k-1}^{(i)} + \zeta_{p^n}^{e_2} s_{k}^{(i)} p_{k-1}^{(i)}}{2p^{n+1}}
\end{align*}
where $\zeta_{p^n}$ is a primitive $p^n$th root of unity.
\end{corollary}
\begin{proof}
Obviously, Assumptions \ref{ass}, \ref{bass} hold for $K = \Q(\sqrt{4\ell})$ and $r_0=r$, so we may apply all of the ideas which culminated in Theorem \ref{mainthm}. In particular, we will exhibit a set of representatives for $(\OO_K/(p^{n+1}))^{\times}$ modulo $\varepsilon$. The assumption that $p$ is inert in $\Q(\sqrt{-\ell})$ is equivalent to the statement that $-\ell$ is not a square modulo $p$. Thus for $t = 1$, $2$, $\ldots$, $p^{n+1}-1$ with $p\nmid t$, we know that $t\sqrt{\ell}$ is a unit modulo $p^{n+1}$ which is never congruent to a power of $\varepsilon$ modulo $p^{n+1}$ since otherwise $-t^2\ell = N(t\sqrt{\ell}) \equiv 1 \pmod{p^{n+1}}$, a contradiction. Now we use our technical assumptions on $\varepsilon$. In the case that $p \equiv 1 \pmod{4}$, we have assumed that $r = p + 1$, so $p$ is inert in $K=\Q(\sqrt{\ell})$ and there is a unique element of order two in $(\OO_K/(p^{n+1}))^{\times} \cong \Z/(p^n(p^2-1))\oplus \Z/(p^n)$ which corresponds to $-1$ and is a power of $\varepsilon$ modulo $p^{n+1}$. Consider the map
\begin{align}\label{1to1}
\left\{ \overline{t}, \overline{t\sqrt{\ell}} : t=1, 2, \ldots, p^{n+1}-1, p \nmid t \right\} \rightarrow C_\m^+
\end{align}
given by the restriction of the map $(\OO_K/(p^{n+1}))^{\times} \rightarrow C_\m^+$ in \ref{exactseq}. This map is two-to-one in the case just described. Similarly, in the case that $p \equiv 3 \pmod{4}$, we have assumed that $r = (p - 1)/2$ is odd, so $p$ is split in $K=\Q(\sqrt{\ell})$ and now the the map in \ref{1to1} is one-to-one since $-1$ is not congruent to a power of $\varepsilon$ modulo $p^{n+1}$ in this case. By Eq.s \ref{C+}, \ref{hterm}, and \ref{prep}, we get for $n$ sufficiently large that $\lambda_p(-\ell)+\lambda_p(-4)$ is
\begin{align*}
&\ord_{\p_n} \mathop{\sum_{t=1}^{p^{n+1}}}_{p\nmid t} \left(\psi_n(t^2)\zeta(0,[(t)]) + \psi_n(-t^2\ell)\zeta(0,[(t\sqrt{\ell})])\right) \\
=&\ord_{\p_n} \sum_{i=1}^{p^n(p-1)} \psi_n(g^2)^i\left(\zeta(0,[(g^i)]) + \psi_n(g)^{e_1}\zeta(0,[(g^i\sqrt{\ell})])\right) \\
=& \ord_{\p_n} \sum_{i=1}^{p^n(p-1)} \frac{\zeta_{p^n}^i}{2} \sum_{k=1}^{p^nrm} \left( b_k\left\{\frac{g^iq_{k-1}}{p^{n+1}}\right\}^2 - 2\left\{\frac{g^iq_{k-1}}{p^{n+1}}\right\}\left\{\frac{g^iq_{k-2}}{p^{n+1}}\right\} \right.\\
&\hspace{0.94 in} +\left. \zeta_{p^n}^{e_2}\left(b_k\left\{\frac{g^ip_{k-1}}{p^{n+1}}\right\}^2 - 2\left\{\frac{g^ip_{k-1}}{p^{n+1}}\right\}\left\{\frac{g^ip_{k-2}}{p^{n+1}}\right\} \right)  \right) \\
&=  \ord_{\p_n} \sum_{i=1}^{p^n(p-1)} \frac{\zeta_{p^n}^i}{2p^{2(n+1)}} \sum_{k=1}^{p^nrm} \left((b_kq_{k-1}^{(i)}-q_{k-2}^{(i)} )q_{k-1}^{(i)} - q_{k-2}^{(i)}q_{k-1}^{(i)}\right.\\
&\hspace{1.4 in} + \zeta_{p^n}^{e_2} \left. \left( (b_kp_{k-1}^{(i)}-p_{k-2}^{(i)} )p_{k-1}^{(i)} - p_{k-2}^{(i)}p_{k-1}^{(i)} \right)\right) \\
&=\ord_{\p_n} \sum_{i=1}^{p^n(p-1)} \frac{\zeta_{p^n}^i}{2p^{2(n+1)}} \sum_{k=1}^{p^nrm} \left(p^{n+1}t_{k}^{(i)}q_{k-1}^{(i)} + q_k^{(i)}q_{k-1}^{(i)} - q_{k-1}^{(i)}q_{k-2}^{(i)}\right.  \\
& \hspace{1.6 in} \left. + \zeta_{p^n}^{e_2}\left(p^{n+1}s_{k}^{(i)}p_{k-1}^{(i)} + p_k^{(i)}p_{k-1}^{(i)} - p_{k-1}^{(i)}p_{k-2}^{(i)}\right)  \right)
\end{align*}
where $\p_n = (1-\zeta_{p^n})$ with $\zeta_{p^n}=\psi_n^2(g)$ and $\psi_n(g)^{e_1} = \psi_n(g)^{2e_2} = \zeta_{p^n}^{e_2}$.

We have the formula $\varepsilon^j(p_{k} + q_{k}\sqrt{\ell}) = p_{jm+k}+q_{jm+k}\sqrt{\ell}$ for all integers $j \geq 0$, $k \geq -1$, so the sequences $p_k$, $q_k$ are periodic modulo $p^{n+1}$ with period $p^nrm$. Thus for all $i\geq 0$, the sequences $p_k^{(i)}$, $q_k^{(i)}$ are periodic with period $p^nrm$, so
\begin{align*}
\sum_{k=1}^{p^nrm} (q_k^{(i)}q_{k-1}^{(i)} - q_{k-1}^{(i)}q_{k-2}^{(i)}) = 0 = \sum_{k=1}^{p^nrm} (p_k^{(i)}p_{k-1}^{(i)} - p_{k-1}^{(i)}p_{k-2}^{(i)}).
\end{align*}
The result follows.
\end{proof}
\begin{rem}
Also, we note that the periodicity also implies that $p_k^{(i)}$, $q_{k}^{(i)}$ are ``palindromic'' in the following sense:
\begin{align*}
p_{k-1}^{(i)} &= p_{p^nrm-k-1}^{(i)} && q_{k-1}^{(i)} = -q_{p^nrm-k-1}^{(i)}.
\end{align*}
As remarked above, the sequence $b_k$ for $k\geq 1$ is periodic with period $m$ and is ``palindromic'' with $b_k = b_{m-k}$ for $1\leq k< m$ while $2b_0 = b_m = b_{2m} = b_{3m} = \ldots$. Hence if $1\leq k \leq p^nrm/2$, then $b_k = b_{p^nrm-k}$, so
\begin{align*}
p^{n+1}s_k^{(i)} &= b_kp_{k-1}^{(i)} -p_{k-2}^{(i)} - p_k^{(i)}  \\
&= b_{p^nrm - k}p_{p^nrm-k-1}^{(i)} -p_{p^nrm-k}^{(i)} - p_{p^nrm-k-2}^{(i)} \\
&= p^{n+1}s_{p^nrm-k}
\end{align*}
and similarly $t_k^{(i)} = - t_{p^nrm-k}^{(i)}$. This implies that we can replace the upper index of the sum on $k$ with $p^nmk/2$ and still maintain the same $\p_n$-adic order. Of course, we could for the same reason ignore the $2$ in the denominator of our sum in the corollary, but it is natural to include this factor of $2$ since
\begin{align*}
\sum_{i=1}^{p^n(p-1)} \zeta_{p^n}^i \sum_{k=1}^{p^nrm/2} \frac{ t_{k}^{(i)}q_{k-1}^{(i)} + \zeta_{p^n}^{e_2} s_{k}^{(i)} p_{k-1}^{(i)}}{2p^{n+1}} \in \Z[\zeta_{p^n}]
\end{align*}
In fact, since the map in \ref{1to1} is two-to-one when $p\equiv 1\pmod{4}$, we can replace the upper index on $i$ with $p^n(p-1)/2$ in this case and still conclude the sum is in $\Z[\zeta_{p^n}]$ with the same $\p_n$-adic order.
\end{rem}
\begin{exam}
Let $p\geq 5$ be a Fermat prime and let $\ell \equiv 3\pmod{4}$ be a prime such that $p$ is inert in $\Q(\sqrt{-\ell})$ and $\Q(\sqrt{\ell})$ has class number $1$. Then $\ell$ is quadratic non-residue modulo $p$, so here we can choose $g=\ell$ assuming additionally that $p^2\nmid \ell^{p-1}-1$. In this case $\zeta_{p^n}=\psi_n(\ell)$ is a primitive $p^n$th root of unity with $\psi_n(g)^2 = \zeta_{p^n}^2$, so we do not have to worry about computing $e_2$ here. For $p$ fixed, these conditions on $\ell$ are just congruence conditions modulo $4p^2$ plus the assumption that $\Q(\sqrt{\ell})$ has class number 1, so there should be many such examples. We just need to check the conditions on the fundamental unit in these cases in order for the corollary to apply.

Let us consider the concrete case of $p=5$ and $\ell=47$. Then $5$ is inert in both $K=\Q(\sqrt{\ell})$ (class number $1$) and $\Q(\sqrt{-\ell})$.  We have
\begin{align*}
\sqrt{47} = [[7; \overline{7, 14}]]
\end{align*}
so $m=2$, the class number of $\Q(\sqrt{-\ell})$ is $(4-3+17-3)/3 = 5$, and the fundamental unit of $K$ is
\begin{align*}
\varepsilon = p_{2-1}+q_{2-1}\sqrt{\ell} = 48+7\sqrt{47}.
\end{align*}
It is easy to check that $\varepsilon$ has order $6 =p+1$ modulo $5$ and that $25\nmid \varepsilon^6-1$. For $n=1$, we compute
\begin{align*}
&\ord_{\p_1} \sum_{i=1}^{10} \zeta_5^{2i} \sum_{k=1}^{30} \frac{ t_{k}^{(i)}q_{k-1}^{(i)} + \zeta_{5} s_{k}^{(i)} p_{k-1}^{(i)}}{50} \\
= &\ord_{\p_1}(6\zeta_{5}^3 + \zeta_{5}^2 + 5\zeta_{5} - 2) = 2 < \varphi(5).
\end{align*}
Hence $\lambda_5(-47)+\lambda_5(-4)=2$. Since $5$ divides the class number of $\Q(\sqrt{-47})$ and $5$ is split in $\Q(i)$, we must have both $\lambda_5(-47)\geq1$ and $\lambda_5(-4)\geq 1$, so $\lambda_5(-4)= 1 = \lambda_5(-47)$.
\end{exam}

\nocite{Wash}

\bibliographystyle{amsalpha}
\bibliography{../References}

\end{document}